\documentclass{article}
\usepackage{amssymb}
\usepackage{amsfonts}
\usepackage{amsmath}
\usepackage{fancyhdr}

\newtheorem{theorem}{Theorem}

\newtheorem{definition}[theorem]{Definition}

\newtheorem{proposition}[theorem]{Proposition}
\newtheorem{remark}[theorem]{Remark}

\newenvironment{proof}[1][Proof]{\noindent\textbf{#1.} }{\ \rule{0.5em}{0.5em}}
\input{tcilatex}
\begin{document}

\title{Volume forms for time orientable Finsler spacetimes\thanks{%
\copyright\ \TEXTsymbol{<}2016\TEXTsymbol{>}. This manuscript version is
made available under the CC-BY-NC-ND 4.0 license
http://creativecommons.org/licenses/by-nc-nd/4.0/}}
\author{Nicoleta Voicu \\
"Transilvania" University, 50, Iuliu Maniu str., Brasov, Romania\\
e-mail: nico.voicu@unitbv.ro}
\date{}
\maketitle

\begin{abstract}
The paper proposes extensions of the notions of Busemann-Hausdorff and
Holmes-Thompson volume to time orientable Finslerian spacetime manifolds.

These notions are designed to also make sense in cases when the Finslerian
metric tensors are either not defined or degenerate along some directions in
each tangent space - which is the case with the majority of Lorentzian
Finsler metrics used in applications. This feature makes it possible to
build well-defined field-theoretical integrals having such metrics as a
background.
\end{abstract}



\section{Introduction}

Finsler spaces represent a natural geometric framework for applications in
physics and biology. Among these, field-theoretical applications have a
peculiar importance and are the most numerous (just a few examples: \cite%
{Asanov}, \cite{Barletta}, \cite{Bogoslovsky}, \cite{Javaloyes}, \cite%
{Kouretsis}, \cite{Laemmerzahl}, \cite{Minguzzi}, \cite{Wohlfarth}, \cite%
{Stavrinos}, \cite{Vacaru}, \cite{Hilbert action}). But, these applications
generally require metrics to be of Lorentzian signature. And, while positive
definite Finsler metrics are quite well studied and understood, there are
basic geometric questions to still be clarified on Finslerian spacetime
metrics, on the answer of which depend most (if not all) field-theoretical
applications.

One of these basic questions (which has so far remained, to our knowledge,
an open one) is the construction of a well-defined volume form for
Lorentz-Finsler manifolds, to be determined from the Finsler metric alone -
and which should reduce to the Riemannian volume form in the particular case
of Riemannian spacetime metrics. The extension of the classical notions of
volume for positive definite Finsler metrics, i.e., the Busemann-Hausdorff
and the Holmes-Thompson ones, \cite{Shen1}, to Lorentzian signature is far
from trivial, for at least two reasons:

- The first problem is that the definitions of these volume forms both
involve the integration of some function on the Finslerian unit balls (or,
equivalently, on the indicatrices) of the given metric. In the case of
positive definite Finsler metrics, these unit balls are compact, leading to
finite integrals and therefore, to well-defined notions, but, in Lorentzian
signature, they become non-compact (just an example: even in the case of the
Minkowski metric $\eta =diag(1,-1,-1,-1)$ on $\mathbb{R}^{4}$, the closed
"unit ball"\ $\{(y^{i})\in \mathbb{R}^{4}~|~\eta _{ij}y^{i}y^{j}\leq 1\}$
is, actually, the interior of a hyperboloid), thus leading to infinite
values of the respective integrals.

- Another serious issue is that, for most of the Lorentzian Finsler
functions used in applications, such as Randers, $m$-th root or Bogoslovsky
ones, there exist entire directions in each tangent space along which the
metric tensor cannot be defined or is degenerate (this issue is even
mentioned in \cite{Skakala-bimetric} as an impediment to building physical
theories based on Finsler metrics).

\bigskip

The paper proposes extensions of the notions of Busemann-Hausdorff and
Holmes-Thompson volume to time orientable Finslerian spacetime manifolds,
meant to also make sense in cases when the corresponding metric tensor is
rather ill-behaved.

The technique is the following. We look for a positive definite Riemannian
metric tensor to be canonically attached to a given Lorentz-Finsler metric
tensor $g$ on a manifold $M$; if such a metric can be found, each of its
closed unit balls $E_{x},$ $x\in M,$ is an ellipsoid (hence, compact) and
can be used, instead of the Finslerian unit balls, in Busemann-Hausdorff and
Holmes-Thompson-type constructions.

This positive definite Riemannian metric is determined as follows:

\textit{Step 1.}\ Each time orientation on $M$ (regarded, as in \cite%
{Torrome}, \cite{Minguzzi}, as a section $x\mapsto t_{x}$ of the tangent
bundle\footnote{%
In the literature, there exist two nonequivalent notions of Finslerian time
orientation, referring either to sections of the pullback bundle$\ \pi
^{\ast }TM,$ \cite{Barletta}, or to sections of $TM,$ \cite{Torrome}, \cite%
{Minguzzi}. Here, it was advantageous to use the latter.} $(TM,\pi ,M)$)
gives rise to a Riemannian spacetime metric $g^{t}$ on $M,$ by the rule $%
g_{x}^{t}:=g(x,t_{x}),$ $\forall x\in M$. Further, using the time
orientation $t$, we can attach to $g^{t}$ a positive definite Riemannian
metric $g^{t,+};$ this is possible using a known trick in general
relativity, \cite{Chocquet-Bruhat}.

\textit{Step 2. }Since the time orientation $t$ is generally not unique, we
can find multiple positive definite metrics $g^{t,+}$ and, accordingly, to
non-unique outcomes for the obtained Busemann-Hausdorff and Holmes-Thompson
type expressions - which is, of course, unacceptable. In order to solve this
ambiguity, we will try to pick a \textit{privileged }time orientation $t_{0}$%
, which provides a minimal critical value $t_{0}$ for the functional 
\begin{equation*}
t\mapsto S_{D}(t)=\underset{D}{\int }\sqrt{\det g^{t,+}(x)}d^{n}x,
\end{equation*}%
providing the $g^{t,+}$-Riemannian volume of an arbitrary compact domain $%
D\subset M.$

Under the assumption that at least a privileged time orientation $t_{0}$
exists on $M$, then the Riemannian volume form $\sqrt{\det g^{t_{0},+}(x)}%
d^{n}x$ does not depend on the choice of the privileged time orientation $%
t_{0}$ and represents a well-defined volume form on $M,$ which we call the 
\textit{minimal Riemannian }volume form. This volume form can actually be
obtained by means of a Busemann-Hausdorff type procedure, in which the role
of the Finslerian unit balls is taken by the Riemannian unit balls $%
E_{x}^{t_{0}}.$

If, moreover, the determinant $\det g$ is smooth and nonzero on the entire
slit tangent bundle, then a Holmes-Thompson-type volume form is also
uniquely defined using a privileged time orientation.

Just as in the positive definite case, the Holmes-Thompson volume is tightly
connected to a certain volume form on the tangent bundle. Thus, it allows
one to naturally define field-theoretical integrals also in the case when
the fields under discussion depend on the fiber coordinates on $TM$; an
alternative construction using the minimal Riemannian volume is also briefly
presented.

In the particular case of Riemannian metrics, all time orientations are
privileged ones and both the above volume forms reduce to the usual
Riemannian one.

\bigskip

The paper is structured as follows. In Section 2, we present some
preliminary notions and results. In Section 3.1, we discuss the notion of
time orientation on the base manifold $M$ and, for each time orientation $t$%
, we construct a positive definite Riemannian metric $g^{t,+}$ from the
initial Lorentz-Finsler metric. In Section 3.2, we introduce the notion of
privileged time orientation. Section 4 is devoted to the introduction of the
minimal Riemannian and of the Holmes-Thompson type\textit{\ }volume forms.
Finally, in the last section, we present three examples: smooth metrics
obtained as linearized Finslerian perturbations of the Minkowski metric $%
diag(1,-1,-1,-1)$ on $\mathbb{R}^{4}$, a non-smooth metric (Berwald-Moor
metric) for which it is still possible to define both volume forms and,
finally, a Bogoslovsky-type metric, for which we can only determine the
minimal Riemannian volume form.

\section{Preliminaries}

\textbf{1. Pseudo-Finsler \textbf{and Finsler }spacetimes}

At present, there exist several different definitions of Finslerian
spacetimes; a recent review thereof is given, e.g., in \cite{Minguzzi}. A
part of them are based on a definition by Asanov, \cite{Asanov}, and rely on
a 1-homogeneous Finslerian fundamental function (norm) $F,$ while the others
are relaxed versions of a definition by Beem, \cite{Beem}, based on a
2-homogeneous function $L.$ In the following, we will prefer the latter,
which give a Finslerian generalization $ds^{2}=L(x,dx)$ of the notion of
relativistic interval and allow one to naturally introduce the notions of
lightlike or spacelike vectors on the base manifold $M$.

The definition of pseudo-Finsler spaces we will present below is the one in 
\cite{Bejancu} and includes all the usual examples of Finslerian spacetime
metrics.

\bigskip

Let $M$ be a connected, orientable, $\mathcal{C}^{\mathcal{\infty }}$-smooth
manifold of dimension $n$ and $(TM,\pi ,M),$ its tangent bundle. The set of
sections of any fibered manifold $E$ over $M$ will be denoted by $\Gamma (E)$
and the set of $\mathcal{C}^{\infty }$-smooth functions on $E,$ by $\mathcal{%
F}(E).$ By $TM^{o},$ we will mean the slit tangent bundle $TM\backslash
\{0\}.$

We denote by $(x^{i})_{i=\overline{0,n-1}}$ the coordinates of a point $x\in
M$ in a local chart $(U,\varphi ).$ Each choice of a basis $\{b_{i}\}$ on $%
T_{x}M$ gives rise to the coordinate $n$-uple $(y^{i})$ for any vector $y\in
T_{x}M$ (the basis $\{b_{i}\}$ can be the natural one $\{\partial /\partial
x^{i}\},$ but this is not necessary, \cite{Bao}, p. 3.). This way, we
obtain, for a point $(x,y)\in \pi ^{-1}(U)\subset TM,$ the coordinates $%
(x^{i},y^{i})_{i=\overline{0,n-1}}.$ Whenever possible, we will make no
distinction between $(x,y)\in TM$ and its coordinates $(x^{i},y^{i})\in 
\mathbb{R}^{2n}.$

In the following, we will only admit positively oriented bases $\{b_{i}\}$
and orientation-preserving coordinate changes $\left( x^{i}\right) \mapsto
(x^{i^{\prime }})$ on $M$, i.e., coordinate changes with $\det (\dfrac{%
\partial x^{i^{\prime }}}{\partial x^{j}})>0.$ We denote by $\{\theta ^{i}\}$
the elements of the dual basis to $\{b_{i}\}.$

\bigskip\ 

Consider a non-empty open submanifold$\ $ $A\subset TM,$ with $\pi (A)=M$
and $0\not\in A.$ We assume that each $A_{x}:=T_{x}M\cap A,$ $x\in M,$ is a
positive conic set, i.e., $\forall \alpha >0,$ $\forall y\in A_{x}:\alpha
y\in A_{x}.$ The set $A$ has the structure of a fibered manifold over $M;$
elements $y\in A_{x}$ are called \textit{admissible vectors}.

\bigskip

Fix a natural number $0\leq q<n.$ A smooth function $L:A\in \mathbb{R}$ is
said to define a \textit{pseudo-Finsler structure} on $M$, if, in any
induced local chart $(\pi ^{-1}(U),\varphi ^{\ast })$ on $TM$ and at any
point $(x,y)\in A\cap \pi ^{-1}(U):$

1)$\ L(x,\alpha y)=\alpha ^{2}L(x,y),$ $\forall \alpha >0$;

2)\ $g_{ij}:=\dfrac{1}{2}\dfrac{\partial ^{2}L}{\partial y^{i}\partial y^{j}}
$ are the components of a quadratic form with $q$ negative eigenvalues and $%
n-q$ positive eigenvalues.

\bigskip

The \textit{Finslerian energy }$L$ can always be prolonged by continuity as $%
0$ at $y=0.$

In particular, if $q=0,$ then the Finsler structure $(M,L)$ is called 
\textit{positive definite.} If $q=n-1,$ then $(M,L)$ is called a \textit{%
Lorentz-Finsler space} or a \textit{Finsler spacetime}. If $A=TM^{o},$ then $%
(M,L)$ is called \textit{smooth. }Usually, by a \textit{Finsler structure},
one automatically understands a smooth, positive definite one, e.g., \cite%
{Bao} - but, here, we will specify this explicitly each time. $(M,L)$ is (%
\textit{pseudo)-Riemannian}, if, in any local chart, $g_{ij}=g_{ij}(x)$ and 
\textit{locally Minkowskian }if around any point of $A,$ there exists a
local chart in which $g_{ij}=g_{ij}(y)$ only.

\bigskip

The arc length of a curve $c:t\in \lbrack a,b]\mapsto (x^{i}(t))$ on $M$ is
calculated as $l(c)=\underset{a}{\overset{b}{\int }}F(x(t),\dot{x}(t))dt,$
where the Finslerian norm $F:A\rightarrow \mathbb{R}$ is given by:%
\begin{equation}
F=\sqrt{\left\vert L\right\vert }.  \label{pos_def_norm}
\end{equation}

The Finslerian metric tensor $g$ can be regarded as a mapping $%
g:A\rightarrow T^{\ast }M\otimes T^{\ast }M$. More precisely, let us fix a
local chart $(U,\varphi )$ on $M$ and $x\in U;$ for each $y=y^{k}b_{k}\in
A_{x},$ we have a symmetric bilinear form $g_{(x,y)}$ on $T_{x}M\simeq 
\mathbb{R}^{n},$ given, in the basis $\{b_{i}\},$ by the matrix $%
g(x,y):=(g_{ij}(x,y))$ i.e.,%
\begin{equation}
g_{(x,y)}(b_{i},b_{j}):=g_{ij}(x,y).  \label{scalar_prod_g}
\end{equation}

\bigskip

Another important quantity in a Finsler space is the \textit{Cartan form }$%
\mathbf{C}=C_{i}(x,y)\theta ^{i},$ with coefficients $C_{i}=\dfrac{1}{2}%
g^{jk}\dfrac{\partial g_{ij}}{\partial y^{k}}\in \mathcal{F}(TM).$ The
coefficients $C_{i}$ are related to $\det (g)$ by:%
\begin{equation}
\dfrac{\partial \sqrt{\det (g)}}{\partial y^{i}}=C_{i}\sqrt{\det (g)}.
\label{Cartan_vector_det}
\end{equation}%
If $(M,L)$ is Riemannian, then $\mathbf{C}$ identically vanishes.

\textbf{Remark. }For smooth, positive definite Finsler metrics, the converse
also holds true, i.e., if $\mathbf{C}$ identically vanishes, then $(M,L)$ is
Riemannian (Deicke's Theorem, \cite{Bao}). Still, for Lorentz-Finsler
metrics, Deicke's Theorem is no longer valid. A counterexample is presented
below, in Section \ref{BM}.

\bigskip

\textbf{2. Volume forms for smooth, positive definite Finsler metrics}

A volume form\textit{\ }$\omega $ on a manifold $M$ is a nowhere zero $n$%
-form on $M:$%
\begin{equation}
\omega =\sigma (x)d^{n}x,  \label{vol_form_M}
\end{equation}%
where $d^{n}x:=dx^{0}\wedge ...\wedge dx^{n-1}.$ With respect to
orientation-preserving coordinate changes $(x^{i})\mapsto (x^{i^{\prime }}),$
the functions $\sigma (x)$ transform as: 
\begin{equation}
\sigma (x)=\det (\dfrac{\partial x^{i^{\prime }}}{\partial x^{j}})\sigma
^{\prime }(x^{\prime }).  \label{vol_transf_M}
\end{equation}

More generally, a volume form can be expressed as a nonzero multiple of the
exterior product $\theta ^{0}\wedge \theta ^{1}\wedge ...\wedge \theta
^{n-1},$ where $\theta ^{i}=\theta ^{i}(x)$ are the elements of an arbitrary
basis of $\Gamma (T^{\ast }M).$ Once a volume form is defined, integrals of
functions on compact domains $D\subset M$ are defined via partitions of
unity.

In particular:\textbf{\ }

- The \textit{Euclidean volume form }on $\mathbb{R}^{n}$ is $d^{n}x.$ The
Euclidean volume of a compact domain $D\subset M$ is denoted by $Vol(D).$

-\ On pseudo-Riemannian manifolds $(M,g),$ the \textit{Riemannian volume form%
} is expressed in an arbitrary basis $\{\theta ^{i}\}$ as, \cite{Fecko},%
\begin{equation}
dV_{g}=\sqrt{\left\vert \det g(x)\right\vert }\theta ^{0}\wedge \theta
^{1}\wedge ...\wedge \theta ^{n-1},  \label{Riemann_vol_element}
\end{equation}%
where $g(x)$ is the matrix of $g$ in the dual basis $\{b_{i}\}$ of $\{\theta
^{i}\}$.

\bigskip

Now, assume that $A=TM^{o}$ and the Finsler structure $(M,L)$ is positive
definite. Fix a local chart $(U,\varphi )$ of $M$ and an arbitrary point $%
x\in U.$ Let $\{b_{i}\}$ be a positively oriented basis of $T_{x}M,$ with
dual $\{\theta ^{i}\}$ and $y=y^{i}b_{i}\in T_{x}M\mapsto (y^{i})\in \mathbb{%
R}^{n},$ the corresponding coordinate isomorphism.

The \textit{closed Finslerian unit ball, }%
\begin{equation}
B_{x}=\{(y^{i})\in \mathbb{R}^{n}~|~F(x,y)\leq 1\}
\label{Finslerian_unit_ball}
\end{equation}%
is a compact, convex subset of $\mathbb{R}^{n}$.

Integrals of homogeneous functions on\textit{\ }a Finslerian unit ball $%
B_{x} $ and its boundary (the \textit{indicatrix})\textit{\ }$\partial B_{x}$
are related, \cite{Crampin-averaging}, as follows. If $f:TM^{o}\rightarrow 
\mathbb{R},$ $(x,y)\mapsto f(x,y)$ is of class $\mathcal{C}^{\infty }$ and
homogeneous of degree $k$ in $y,$ then 
\begin{equation*}
\underset{B_{x}}{\int }f(x,y)d^{n}y:=~\underset{\varepsilon \rightarrow 0}{%
\lim }\underset{\varepsilon \leq F(x,y)\leq 1}{\int }f(x,y)d^{n}y
\end{equation*}%
is well defined and%
\begin{equation}
\underset{\partial B_{x}}{\int }f\lambda =(n+k)\underset{B_{x}}{\int }%
f(x,y)d^{n}y,  \label{sphere_to_ball_rel}
\end{equation}%
where $\lambda $ denotes the Euclidean volume form on $\partial B_{x}.$

\bigskip

The \textbf{Busemann-Hausdorff volume form} of $(M,L)$ is given, in the
basis $\{\theta ^{i}\}$ of $\Gamma (T^{\ast }M)$, \cite{Shen1}, as: 
\begin{equation}
dV_{BH}=\sigma _{BH}(x)\theta ^{0}\wedge ...\theta ^{n-1},~\ \ \sigma
_{BH}(x)=\dfrac{Vol(\mathbb{B})}{Vol(B_{x})},  \label{BH_volume_form}
\end{equation}%
where $\mathbb{B}$ denotes the Euclidean unit ball in $\mathbb{R}^{n}.$

In the case when $L$ is reversible, i.e., $L(x,y)=L(x,-y),$ $\forall
(x,y)\in TM^{o},$ $dV_{BH}$ gives the Hausdorff measure of the distance
function induced by $F=\sqrt{L}.$

\bigskip

The \textbf{Holmes-Thompson volume form }of $(M,L)$ is, \cite{Shen1}:%
\begin{equation}
dV_{HT}=\sigma _{HT}(x)\theta ^{0}\wedge ...\theta ^{n-1},~\ \ \sigma
_{HT}(x)=\dfrac{1}{Vol(\mathbb{B})}\underset{B_{x}}{\int }\det g(x,y)d^{n}y,
\label{HT_volume_form}
\end{equation}%
with $d^{n}y:=dy^{0}\wedge dy^{1}\wedge ...\wedge dy^{n-1}.$

\bigskip

\textbf{Particular case: }If\textbf{\ }$(M,g)\ $is Riemannian, then: $%
dV_{BH}=dV_{HT}=dV_{g}.$

\bigskip

\textbf{Remark.} The fact that, for smooth, positive definite Finsler
metrics, the unit balls $B_{x},$ $x\in M,$ are compact, is essential for
both the above definitions. Also, (\ref{HT_volume_form}) uses the fact that $%
\det (g)$ is defined (and smooth) on the entire $TM^{o}.$

\section{Time orientable Finsler spacetimes}

\subsection{Time orientations and osculating Riemannian metrics}

Assume, in the following, that $(M,L)$ is a Lorentz-Finsler manifold. An
admissible tangent vector $y\in A_{x}$ at some point $x\in M$ is called, 
\cite{Torrome}, \cite{Minguzzi}: a)\ \textit{timelike, }if $L(x,y)>0;$ b)\ 
\textit{lightlike, }if $L(x,y)=0$ and c) \textit{spacelike}, if $L(x,y)<0.$

If $y\in A_{x}$ is timelike, then $L(x,y)=F^{2}(x,y).$

\begin{definition}
\cite{Torrome}: A \textit{time orientation} on $M$ is a smooth vector field $%
t\in \Gamma (A),$ which is everywhere timelike.
\end{definition}

If the Finsler spacetime $(M,L)$ admits a time orientation, then it is
called \textit{time orientable}.

In the following, we will always assume that $(M,L)$ is time orientable%
\footnote{%
The vast majority of "reasonable"\ physical models are based on time
orientable spacetimes. The case of non-time-orientable spacetimes can be
dealt with, e.g., by passing to a double covering space, \cite{Hawking}.}.

Consider the subset of $A$ consisting of timelike vectors:%
\begin{equation}
A^{+}:=A\cap L^{-1}(0,\infty ).  \label{A_plus}
\end{equation}%
The open subset $A^{+}\subset A$ is a submanifold of $A$. Moreover, since $%
(M,L)\ $is time orientable, there exists at every $x\in M,$ at least a
timelike vector, i.e., $\pi (A^{+})=M.$ Consequently, $A^{+}$ has the
structure of a fibered manifold over $M$. Time orientations $t$ can thus be
regarded as (smooth) sections of $A^{+}.$

\bigskip

Let $t\in \Gamma (A^{+}),$ $x\mapsto t_{x}$ denote an arbitrary time
orientation on $M.$ Then, the mapping $x\in M\mapsto g_{x}^{t}\in T^{\ast
}M\otimes T^{\ast }M,$ given by: 
\begin{equation}
g_{x}^{t}:=g_{(x,t_{x})},~\ \ \forall x\in M,  \label{osculating_g}
\end{equation}%
defines a pseudo-Riemannian metric $g^{t}$ on $M$, called, \cite{Stavrinos},
an \textit{osculating Riemannian metric}\footnote{%
The term \textit{osculating Riemannian metric} is also used in the Finsler
literature (e.g., \cite{Crampin-averaging}) with a different meaning, i.e.,
a Riemannian metric obtained by an averaging technique. Here, we do \textit{%
not }have in mind this meaning (averaging techniques are not even available
until a volume form is defined).} of the Lorentz-Finsler metric $g.$

\bigskip

\textbf{Remark. }The fact that $t$ is everywhere admissible ensures that $%
g_{x}^{t}(v,w)$ is well defined for any vectors $v,w\in T_{x}M$ and the
dependence $x\mapsto g_{x}^{t}$ is smooth (even if $A\varsubsetneq TM^{o},$
i.e., if the initial Finsler metric $g=g(x,y)$ is ill-behaved along certain
directions $y\in T_{x}M$). That is, $g^{t}$ is, indeed, a well-defined
pseudo-Riemannian metric.

\bigskip

Now, fix a time orientation $t\in \Gamma (A^{+}).$\textbf{\ }The osculating
Riemannian metric $g^{t},$ $t\in \Gamma (A)$ has Lorentzian signature $%
(+,-,-,...,-)$. Following the model in \cite{Chocquet-Bruhat} (Remark 2.4,
Ch. XII), we define the mapping $g^{t,+}:M\rightarrow T^{\ast }M\otimes
T^{\ast }M,~~x\mapsto g_{x}^{t,+},$ with\footnote{%
The signs in (\ref{g_plus}) differ from the ones in \cite{Chocquet-Bruhat}
due to different metric signature conventions.}: 
\begin{equation}
g_{x}^{t,+}(v,w):=2g_{x}^{t}(t_{x}^{\prime },v)g_{x}^{t}(t_{x}^{\prime
},w)-g_{x}^{t}(v,w),~\ \ \ \ \ \ \ \forall v,w\in T_{x}M,  \label{g_plus}
\end{equation}%
where $t^{\prime }:=\dfrac{t}{F(t)}$. In local writing, we have, at any $%
x\in M:$%
\begin{equation}
g_{ij}^{t,+}=2t_{i}^{\prime }t_{j}^{\prime }-g_{ij}^{t},
\label{g_plus_coords}
\end{equation}%
where $t_{i}^{\prime }=g_{ij}^{t}t^{\prime j}.$

\begin{proposition}
i) Given any time orientation $t\in \Gamma (A^{+}),$ the metric $g^{t,+}$ is
a positive definite Riemannian metric on $M$.

ii) Corresponding to any local chart, there holds the equality:%
\begin{equation}
\det (g^{t,+})=\left\vert \det (g^{t})\right\vert .  \label{determinants}
\end{equation}
\end{proposition}

\begin{proof}
\textit{i)}\ Each $g_{x}^{t,+},$ $x\in M,$ is a symmetric bilinear form on $%
T_{x}M$ and the dependence $x\mapsto g_{x}^{t,+}$ is smooth, i.e., $g^{t,+}$
is a pseudo-Riemannian metric.

Let us also check that each $g_{x}^{t,+}$ is positive definite. Fix an
arbitrary $x\in M$ and pick a $g_{x}^{t}$-orthonormal basis $\{\hat{e}%
_{i}\}_{i=\overline{0,n-1}}$ on $T_{x}M$, with $\hat{e}_{0}=t^{\prime };$
that is, $g_{x}^{t}(\hat{e}_{i},\hat{e}_{j})=\eta _{ij},$ where$\ \eta
=diag(1,-1,-1,...,-1)$. For any $v=v^{i}\hat{e}_{i}\in T_{x}M,$ we have: $%
g_{x}^{t}(t^{\prime },v)=g_{x}^{t}(\hat{e}_{0},v)=v^{0},$ $g^{t}(v,v)=\left(
v^{0}\right) ^{2}-\left( v^{1}\right) ^{2}-...-\left( v^{n-1}\right) ^{2}$
and: 
\begin{equation*}
g_{x}^{t,+}(v,v)=2[g_{x}^{t}(t_{x}^{\prime },v)]^{2}-g_{x}^{t}(v,v)=\left(
v^{0}\right) ^{2}+\left( v^{1}\right) ^{2}+...+\left( v^{n-1}\right)
^{2}\geq 0.
\end{equation*}%
Moreover, $g_{x}^{t,+}(v,v)=0$ if and only if $v^{i}=0,$ $i=\overline{0,n-1}%
, $ i.e., $v=0.$

\textit{ii)} We will use the following result, \cite{Bao}, p. 287: If $%
(Q_{ij})$ is a nonsingular $n\times n$ complex matrix with inverse $(Q^{ij})$
and $C_{j}\in \mathbb{C},$ $j=\overline{0,n-1},$ then:%
\begin{equation}
\det (Q_{jk}+C_{j}C_{k})=(1+Q^{hl}C_{h}C_{l})\det (Q_{jk}).
\label{matrix_result}
\end{equation}%
Fix an arbitrary point $x\in M,$ a local chart around $x$ and a basis $%
\{b_{i}\}$ of $T_{x}M$. Set: $Q_{jk}=g_{jk}^{t}(x),$ $C_{j}=i\sqrt{2}%
t_{j}^{\prime }(x)\in \mathbb{C}.$ Taking into account (\ref{g_plus_coords}%
), we have, at $x$: 
\begin{equation*}
\det (g^{t,+})=(-1)^{n}\det (g_{ij}^{t}-2t_{i}^{\prime }t_{j}^{\prime
})=(-1)^{n}(1-2(g^{t})^{ij}t_{i}^{\prime }t_{j}^{\prime })\det (g^{t})
\end{equation*}%
With $(g^{t})^{ij}t_{i}^{\prime }t_{j}^{\prime }=L(t^{\prime })=1,$ we get $%
\det (g^{t,+})=(-1)^{n+1}\det (g^{t})=\left\vert \det (g^{t})\right\vert $.
\end{proof}

\bigskip

Fix $x\in M$, a basis $\{b_{i}\}$ of $T_{x}M$ and denote by $(y^{i})$ the
coordinates of $y\in T_{x}M$ in this basis. If $t$ is a time orientation on $%
M,$ the closed unit ball of $g^{t,+}$ at $x$ is the ellipsoid 
\begin{equation}
E_{x}^{t}=\{(y^{i})\in \mathbb{R}^{n}~|~g_{x}^{t,+}(y,y)\leq 1\}.
\label{osc_metric_ellipsoid}
\end{equation}

The ellipsoid (\ref{osc_metric_ellipsoid}) is the image of the Euclidean
unit ball $\mathbb{B}=\{(u^{i^{\prime }})\in \mathbb{R}^{n}~|~\delta
_{i^{\prime }j^{\prime }}u^{i^{\prime }}u^{j^{\prime }}\leq 1\}$ through a
linear transformation $\varphi :\mathbb{R}^{n}\rightarrow \mathbb{R}%
^{n},(u^{i^{\prime }})\mapsto (y^{i}),$ given by:%
\begin{equation}
y^{i}=a_{j^{\prime }}^{i}u^{j^{\prime }},  \label{linear_coord_change}
\end{equation}%
with Jacobian determinant:%
\begin{equation}
\det (a_{j^{\prime }}^{i})=[\det g^{t,+}(x)]^{-1/2}>0.
\label{Jacobian_lin_transf}
\end{equation}%
(The matrix $(a_{j^{\prime }}^{i})$ is determined as the matrix of change of
basis from the initial basis $\{b_{i}\}$ to a positively oriented, $g^{t,+}$%
-orthonormal basis $(e_{i}^{\prime })$ with corresponding coordinates
denoted by $(u^{i^{\prime }})\in \mathbb{R}^{n}$).

As a consequence, Euclidean volume $Vol(E_{x}^{t})=\underset{E_{x}^{t}}{\int 
}d^{n}y$ is:%
\begin{equation}
Vol(E_{x}^{t})=\dfrac{Vol(\mathbb{B})}{\sqrt{\det g^{t,+}(x)}}=\dfrac{Vol(%
\mathbb{B})}{\sqrt{\left\vert \det g^{t}(x)\right\vert }}.  \label{volume_Et}
\end{equation}

\subsection{Privileged time orientation}

The time orientation of a given Lorentz-Finsler manifold is, generally, far
from unique - and different time orientations $t\in \Gamma (A^{+})$ give
rise to different Riemannian volume forms $dV_{g^{t}}$. In the following, we
will try to pick a time orientation which, roughly speaking, \textit{%
minimizes} the Riemannian volume of an arbitrary compact domain $D\subset X$.

\bigskip

Fix an arbitrary compact domain $D\subset M.$ The functional $S_{D}:\Gamma
(A^{+})\rightarrow \mathbb{R},$ defined by:%
\begin{equation}
S_{D}(t):=\underset{D}{\int }dV_{g^{t,+}}=~\underset{D}{\int }\sqrt{%
\left\vert \det g_{x}^{t}\right\vert }d^{n}x,  \label{Finsler_Riemann_vol}
\end{equation}%
where $g_{ij}^{t}(x)=g_{x}^{t}(\dfrac{\partial }{\partial x^{i}},\dfrac{%
\partial }{\partial x^{j}})$, is invariant to arbitrary coordinate changes
on $M.$ Also, $S_{D}$ admits an infimum (as $S_{D}(t)>0,\forall t\in \Gamma
(A^{+})$).

A global minimum for $S_{D}$ on the (open) set $A^{+}$ is not guaranteed to
exist, as $S_{D}$ might decrease to its infimum as we approach, e.g., a
lightlike direction (see, e.g., the example in Subsection \ref%
{Bogoslovsky_metric}). This is why, in order to increase our chances of
obtaining well-defined volume forms on $(M,L),$ we will relax the minimality
request, as follows.

\begin{definition}
\label{privileged}We call a privileged time orientation on $(M,L)$, any time
orientation $t_{0}\in \Gamma (A^{+})$, such that $S_{D}(t_{0})$ is a minimum
of the set of critical values of the functional $S_{D}$ in (\ref%
{Finsler_Riemann_vol}), for any compact domain $D\subset M$.
\end{definition}

\textbf{Particular case }\textit{(Riemannian spacetimes)}\textbf{:} If $%
g=g(x)$ only, then the mapping $t\mapsto S_{D}(t)$ is, in fact, constant. In
other words, \textit{in time orientable Riemannian spacetimes, all time
orientations are privileged ones.}

\bigskip

Let us determine the Euler-Lagrange equations for $S_{D}.$ As the Lagrangian
density $\mathcal{L}:=\sqrt{\left\vert \det g(x,t)\right\vert }$ in (\ref%
{Finsler_Riemann_vol}) is of order zero in $t$, critical points of $S_{D}$
are given by:%
\begin{equation}
\dfrac{\partial }{\partial t^{i}}\sqrt{\left\vert \det g(x,t)\right\vert }%
=0,~\ \ \forall x\in M.  \label{zero_deriv_gt}
\end{equation}%
Using (\ref{Cartan_vector_det}) and the fact that $\det g(x,t)\not=0,$ the
above is equivalent to%
\begin{equation}
C_{i}(x,t)=0.  \label{Zero_Cartan_coords}
\end{equation}%
This can be reformulated as:

\begin{proposition}
If $t_{0}\in \Gamma (A^{+})$ is a critical point of the functional $S_{D}$,
then it is a zero of the Cartan form: 
\begin{equation}
\mathbf{C}(x,t_{0}(x))=0,~\ \forall x\in M.  \label{zero_Cartan_coord_free}
\end{equation}
\end{proposition}

\section{Volume forms}

\subsection{Minimal Riemannian volume form}

Assume, in the following, that the Lorentz-Finsler manifold $(M,L)$ admits
at least a privileged time orientation $t_{0}.$

Let us start with the following remark. If $t_{0}$ and $t_{0}^{\prime }$ are
two privileged time orientations, then, for any compact $D\subset M,$ they
provide the same (minimal critical) value for $S_{D}.$ In other words, the
values $S_{D}(t_{0})$ and $S_{D}(t_{0}^{\prime })$ have to coincide for 
\textit{all }compact domains $D\subset M.$ As a consequence, the
corresponding Lagrangian densities (which are smooth functions) have to
coincide pointwise, i.e., for any $x\in M$ and in any local chart around $x$%
, 
\begin{equation}
\sqrt{\left\vert \det g^{t_{0}}(x)\right\vert }=\sqrt{\left\vert \det
g^{t_{0}^{\prime }}(x)\right\vert }.  \label{equal_dets}
\end{equation}%
That is, 
\begin{equation}
dV_{g^{t_{0}}}=dV_{g^{t_{0}^{\prime }}}.  \label{volume_forms_equality1}
\end{equation}%
It makes thus sense

\begin{definition}
We call the \textbf{minimal Riemannian volume form }on $(M,L),$ the
differential form:%
\begin{equation}
dV_{bh}=dV_{g^{t_{0}}},  \label{minimal_R_volume}
\end{equation}%
where $t_{0}$ is any privileged time orientation for $(M,L).$
\end{definition}

Relation (\ref{volume_forms_equality1}) ensures that the above definition
does not depend on the choice of the privileged time orientation $t_{0}.$

\bigskip

\textbf{Particular case.} If $(M,g)$ is Riemannian, then $dV_{bh}$ coincides
with the Riemannian volume form $dV_{g}.$

\bigskip

The minimal Riemannian volume form can be also obtained via a
Busemann-Hausdorff type construction, as follows. Let us denote by $%
\{b_{i}\} $ a positively oriented basis of $\Gamma (TM)$ and by $\{\theta
^{i}\},$ the dual basis of $\Gamma (T^{\ast }M).$ In this basis, the minimal
Riemannian volume element is expressed as $dV_{bh}=\sqrt{\left\vert \det
g^{t_{0}}(x)\right\vert }\theta ^{0}\wedge ...\wedge \theta ^{n-1}.$ Using (%
\ref{volume_Et}), we have%
\begin{equation}
\sqrt{\left\vert \det g^{t}(x)\right\vert }=\dfrac{Vol(\mathbb{B})}{%
Vol(E_{x}^{t})},
\end{equation}%
which leads to:%
\begin{equation}
dV_{bh}=\sigma _{bh}(x)\theta ^{0}\wedge \theta ^{1}\wedge ...\wedge \theta
^{n-1},~~\ \ \ \sigma _{bh}(x)=\dfrac{Vol(\mathbb{B})}{Vol(E_{x}^{t_{0}})}.
\label{Busemann-Hausdorff_expr}
\end{equation}

Equation (\ref{Zero_Cartan_coords}), together with (\ref{volume_Et}) tell us
that, at each point $x$ of the base manifold, and in any local chart around $%
x$, the volume $Vol(E_{x}^{t_{0}})$ is a critical value of the real-valued
mapping $t_{x}\mapsto Vol(E_{x}^{t})$.

\subsection{Holmes-Thompson volume form}

Consider a privileged time orientation $t_{0}\in \Gamma (A^{+})$ on $M$ and
denote by $E_{x}^{t_{0}}$ the corresponding Riemannian unit ball at $x\in M.$
We assume, in the following, that the function $\det g=\det g(x,y)$ (where,
by $g_{ij}(x,y)$, we mean $g_{(x,y)}(b_{i},b_{j})$) can be continuously
prolonged to $TM^{o}$ and the prolongation, also denoted by $\det g$,\ is
smooth and nowhere zero. Under these circumstances, it makes sense:

\begin{definition}
The \textbf{Holmes-Thompson volume form} on the Lorentz-Finsler space $(M,L)$
is the differential form:%
\begin{equation}
dV_{ht}=\sigma _{ht}(x)\theta ^{0}\wedge \theta ^{1}\wedge ...\wedge \theta
^{n-1},~\ \ \sigma _{ht}(x)=\dfrac{1}{Vol(\mathbb{B})}\underset{E_{x}^{t_{0}}%
}{\int }\left\vert \det g(x,y)\right\vert d^{n}y.
\label{HT_volume_Lorentzian}
\end{equation}
\end{definition}

\begin{proposition}
\label{HT_Prop}$dV_{ht}$ is a well-defined volume form on $M.$
\end{proposition}

\begin{proof}
\textit{1. Nondegeneracy: }Fix an arbitrary atlas on $TM$. Taking into
account that the function $\left\vert \det (g)\right\vert $ is smooth on $%
TM^{o}$ and 0-homogeneous in $y,$ the integral (\ref{HT_volume_Lorentzian})
is well defined and it can be expressed, using (\ref{sphere_to_ball_rel}),
as an integral on the boundary $\partial E_{x}^{t_{0}}$:%
\begin{equation*}
\underset{E_{x}^{t_{0}}}{\int }\left\vert \det g(x,y)\right\vert d^{n}y=%
\dfrac{1}{n}\underset{\partial E_{x}^{t_{0}}}{\int }\left\vert \det
g(x,y)\right\vert \lambda .
\end{equation*}%
As $\partial E_{x}^{t_{0}}$ is compact and $\left\vert \det g\right\vert $
is continuous on $\partial E_{x}^{t_{0}}$, the minimum 
\begin{equation*}
\underset{(y^{i})\in \partial E_{x}^{t}}{\min }\left\vert \det
g(x,y)\right\vert =:g_{\min }(x)
\end{equation*}%
always exists - and, under the above assumptions, it is strictly positive.
Taking, again, into account the 0-homogeneity of $g,$ we have: $g_{\min }(x)=%
\underset{(y^{i})\in E_{x}^{t}}{\min }\left\vert \det g(x,y)\right\vert ,$
therefore,%
\begin{equation}
\sigma _{ht}(x)\geq \dfrac{g_{\min }(x)}{Vol(\mathbb{B})}\underset{%
E_{x}^{t_{0}}}{\int }d^{n}y=g_{\min }(x)\dfrac{Vol(E_{x}^{t_{0}})}{Vol(%
\mathbb{B})}>0,~\   \label{ineq_ht_bh}
\end{equation}%
which proves the statement.

\textit{2. The rule of transformation} (\ref{vol_transf_M}) with respect to
coordinate changes $(x^{i})\mapsto (x^{i^{\prime }})$ on $M:$

Fix $x\in M.$ A brief computation using (\ref{scalar_prod_g}) shows that $%
dV_{ht}$ is invariant to changes of bases $\{\theta ^{i}\}\rightarrow
\{\theta ^{i^{\prime }}\}$ on each cotangent space $T_{x}^{\ast }M$. Hence,
we can take with no loss of generality $\theta ^{i}:=dx^{i},$ i.e., in (\ref%
{HT_volume_Lorentzian}), $dV_{ht}=\sigma _{ht}(x)d^{n}x$ and $%
g_{ij}=g_{(x,y)}(\dfrac{\partial }{\partial x^{i}},\dfrac{\partial }{%
\partial x^{j}}).$

Consider an arbitrary coordinate change $(x^{i})\mapsto (x^{i^{\prime }})$
on $M.$ Relative to the induced coordinate change: $x^{i}=x^{i}(x^{k^{\prime
}}),$ $y^{i}=\dfrac{\partial x^{i}}{\partial x^{k^{\prime }}}y^{k^{\prime }}$
on $TM$, the functions $g_{ij}$ transform as: $g_{ij}=\dfrac{\partial
x^{i^{\prime }}}{\partial x^{i}}\dfrac{\partial x^{j^{\prime }}}{\partial
x^{j}}g_{i^{\prime }j^{\prime }};$ therefore, $\det g(x,y)=[\det (\dfrac{%
\partial x^{i^{\prime }}}{\partial x^{j}})]^{2}\det g(x^{\prime },y^{\prime
}),$ which, substituted into (\ref{HT_volume_Lorentzian}), gives the result.

\textit{3. Independence on the choice of the privileged time orientation }$%
t_{0}$\textit{: }Let $t_{0},\tilde{t}_{0}\in \Gamma (A^{+})$ be two
privileged time orientations. Fix a local chart $(U,\varphi )$ on $M$, an
arbitrary point $x\in U$ and a basis $\{b_{i}\}$ of $T_{x}M.$ Then, each of
the ellipsoids $E_{x}^{t_{0}},$ $E_{x}^{\tilde{t}_{0}}$ is the image of $%
\mathbb{B}$ through an invertible linear transformation, as in (\ref%
{linear_coord_change}). More precisely, take: $\varphi ,\tilde{\varphi}:%
\mathbb{R}^{n}\rightarrow \mathbb{R}^{n},$ given by: $u^{i^{\prime }}\mapsto
y^{i}:=a_{j^{\prime }}^{i}u^{j^{\prime }}$ and $u^{i^{\prime }}\mapsto 
\tilde{y}^{i}:=\tilde{a}_{j^{\prime }}^{i}u^{j^{\prime }}$ respectively,
such that 
\begin{equation*}
\varphi (\mathbb{B})=E_{x}^{t_{0}},~\ \ \tilde{\varphi}(\mathbb{B})=E_{x}^{%
\tilde{t}_{0}}.
\end{equation*}%
The corresponding Jacobian determinants are as in (\ref{Jacobian_lin_transf}%
), i.e.,%
\begin{equation*}
\det (a_{j^{\prime }}^{i})=\left\vert \det g^{t_{0}}(x)\right\vert
^{-1/2},~\ \det (\tilde{a}_{j^{\prime }}^{i})=\left\vert \det g^{\tilde{t}%
_{0}}(x)\right\vert ^{-1/2}.
\end{equation*}%
Since $t_{0}$ and $\tilde{t}_{0}$ are both privileged time orientations, we
have: $\left\vert \det g^{t_{0}}(x)\right\vert ^{-1/2}=\left\vert \det g^{%
\tilde{t}_{0}}(x)\right\vert ^{-1/2}.$ Therefore, the linear mapping%
\begin{equation}
\varphi \circ \tilde{\varphi}^{-1}:\mathbb{R}^{n}\mapsto \mathbb{R}^{n},~(%
\tilde{y}^{i})\mapsto (y^{j})  \label{diffeo_Et0}
\end{equation}%
is volume-preserving, i.e., $\det (\dfrac{\partial y^{i}}{\partial \tilde{y}%
^{j}})=1,$ and maps diffeomorphically $E_{x}^{\tilde{t}_{0}}$ to $%
E_{x}^{t_{0}}.$ Taking into account that, by their definition, the functions 
$g_{ij}(x,y)$ behave tensorially under \textit{linear} transformations $(%
\tilde{y}^{i})\mapsto (y^{j})$ on $\mathbb{R}^{n}$ (which can be traced back
to changes of bases $\{b_{i}\}\rightarrow \{\tilde{b}_{i}\}$ on $T_{x}M$),
we have: $\det g(x,y)=[\det (\dfrac{\partial \tilde{y}^{i}}{\partial y^{j}}%
)]^{2}\det g(x,\tilde{y})=\det g(x,\tilde{y})$ and%
\begin{equation*}
\dfrac{1}{Vol(\mathbb{B})}\underset{E_{x}^{t_{0}}}{\int }\left\vert \det
g(x,y)\right\vert d^{n}y=\dfrac{1}{Vol(\mathbb{B})}\underset{E_{x}^{\tilde{t}%
_{0}}}{\int }\left\vert \det g(x,\tilde{y})\right\vert d^{n}\tilde{y},
\end{equation*}%
i.e. $\sigma _{ht}(x)$ does not depend on the choice of the privileged time
orientation.
\end{proof}

\bigskip

\textbf{Particular case.} If $(M,L)$ is Riemannian, then, using (\ref%
{volume_Et}), we get:%
\begin{equation*}
\sigma _{ht}(x)=\dfrac{\left\vert \det g(x)\right\vert }{Vol(\mathbb{B})}%
Vol(E_{x}^{t})=\sqrt{\left\vert \det g(x)\right\vert },
\end{equation*}%
i.e., the Holmes-Thompson volume form (\ref{HT_volume_Lorentzian}) reduces
to the Riemannian volume form $dV_{g}.$

\bigskip

\textbf{Field-theoretical integrals with direction dependent fields. }Just
as in the positive definite case (\cite{Shen1}, p. 26), the Holmes-Thompson
volume form is tightly related to a volume form on $TM$. This allows us to
naturally introduce field-theoretical actions in the case when the fields
also depend on the directional variables $y^{i},$ i.e., they are represented
by sections $(x,y)\mapsto q^{\sigma }(x,y)$ of some fibered manifold over $%
TM $.

Consider a smooth Lagrangian function 
\begin{equation*}
\mathfrak{L}(x,y):\mathfrak{=L}(x,y,q^{\sigma }(x,y),q_{,i}^{\sigma
}(x,y),q_{\cdot i}^{\sigma }(x,y),....,q_{,i_{1}...\cdot i_{r}}^{\sigma
}(x,y))
\end{equation*}%
on $TM$ (where $_{,i}$ and $_{\cdot i}$ denote partial differentiation with $%
x^{i}$ and $y^{i}$ respectively), which is invariant under arbitrary
coordinate changes on $TM.$ The \textit{action} attached to $\mathfrak{L}$
and to a compact domain $D\subset M$ can be defined\footnote{%
A somewhat similar expression of a Finslerian action is to be found in \cite%
{Wohlfarth}; the diference is that, in the cited paper, integration of the
Lagrangian with respect to the fiber coordinates $y^{i}$ is carried out on
the indicatrices $S_{x}$ (given by $\left\vert L\right\vert =1$) of the
initial Lorentz-Finsler metric - which are non-compact, thus leading to
improper integrals. Here, these indicatrices are replaced by the compact
sets $E_{x}^{t_{0}}.$} as:%
\begin{equation}
S_{D}(q)=\dfrac{1}{Vol(\mathbb{B})}\underset{D}{\int }[\underset{%
E_{x}^{t_{0}}}{\int }\mathfrak{L}(x,y)\left\vert \det g(x,y)\right\vert
d^{n}y]\theta ^{0}\wedge ...\wedge \theta ^{n-1}.  \label{action_TM}
\end{equation}%
By a similar reasoning to the one in Proposition \ref{HT_Prop}, we find that
the value $S_{D}(q)$ does not depend on the choice of the privileged time
orientation $t_{0}$.

\begin{remark}
If the determinant $\det (g)$ cannot be continuously prolonged by nonzero
values to the entire slit tangent bundle $TM^{o},$ then (\ref{action_TM})
cannot be constructed. In this case, we can still obtain a well-defined
action if we replace in (\ref{action_TM}), $\left\vert \det
g(x,y)\right\vert ,$ by $\left\vert \det (g^{t_{0}}(x))\right\vert .$
\end{remark}

\section{Examples}

\subsection{\label{linearized}Linearized perturbations of the Minkowski
metric}

Consider, on the Minkowski spacetime $(\mathbb{R}^{4},\eta )$ (where $\eta
=diag(1,-1,-1,-1)$), an arbitrary smooth, positive definite Finsler metric
tensor $\gamma _{ij}=\gamma _{ij}(x,y)$ and a small constant $\varepsilon
>0, $ with $\varepsilon ^{2}\simeq 0.$ We define, on $T\mathbb{R}%
^{4}\backslash \{0\},$ the function:%
\begin{equation}
L(x,y)=\eta _{ij}y^{i}y^{j}+\varepsilon \gamma _{ij}(x,y)y^{i}y^{j}.
\label{linearized_L}
\end{equation}%
This gives a smooth Lorentz-Finsler structure on $\mathbb{R}^{4}$, with
metric tensor $g_{ij}(x,y)=\eta _{ij}+\varepsilon \gamma _{ij}(x,y).$ Its
determinant%
\begin{equation}
\left\vert \det (g(x,y))\right\vert =1+\varepsilon \eta ^{ij}\gamma
_{ij}(x,y)  \label{linearized_g}
\end{equation}%
is 0-homogeneous in $y$ and smooth on $T\mathbb{R}^{4}\backslash \{0\},$
hence, it admits a nonzero global minimum on each tangent space $T_{x}%
\mathbb{R}^{4}$. Privileged time orientations $t_{0}\in \Gamma (T\mathbb{R}%
^{4}\backslash \{0\})$ are solutions of (\ref{zero_deriv_gt}), i.e., $\eta
^{ij}\gamma _{ij\cdot k}(t)=0.$ Once a privileged time orientation is
chosen, we can write:%
\begin{equation*}
dV_{bh}=\sqrt{\left\vert \det g(x,t_{0,x})\right\vert }d^{4}x
\end{equation*}%
and $dV_{ht}$ is given by (\ref{HT_volume_Lorentzian}).

\subsection{Berwald-Moor metric\label{BM}}

Consider, on $M=\mathbb{R}^{4},$ a sign-adjusted version of the \textit{%
Berwald-Moor} quartic Finslerian metric, \cite{Bogoslovsky}:%
\begin{equation*}
L=sgn(y^{0}y^{1}y^{2}y^{3})\sqrt{\left\vert y^{0}y^{1}y^{2}y^{3}\right\vert }%
.
\end{equation*}%
The sign $sgn(y^{0}y^{1}y^{2}y^{3})$ (which does not appear in \cite%
{Bogoslovsky}) is introduced in order to allow $L$ to also take negative
values - and hence, to be able to define $L$-spacelike vectors.

The corresponding metric tensor%
\begin{equation}
\left( g_{ij}\right) =\left\{ 
\begin{array}{c}
-\dfrac{1}{8}\dfrac{L}{(y^{i})^{2}},~\ i=j\,\  \\ 
\dfrac{1}{8}\dfrac{L}{y^{i}y^{j}},~\ i\not=j%
\end{array}%
\right.  \label{mBM}
\end{equation}%
is only defined on $T\mathbb{R}^{4}\backslash \{y~|~\exists i:y^{i}=0\}$ -
and tends to infinity as we approach any of the hyperplanes $y^{i}=0$.
Still, its determinant 
\begin{equation}
\det (g_{ij}(y))=-2^{-8},~\ \forall (y^{i})\in \mathbb{R}^{4}
\label{det_B_M}
\end{equation}%
is a constant and hence, admits a smooth prolongation to the entire $T%
\mathbb{R}^{4}\backslash \{0\}.$ Any time orientation $t$ is a privileged
one and gives the same value $\sqrt{\left\vert \det (g^{t})\right\vert }%
=2^{-4}.$ Substituting into (\ref{minimal_R_volume}), the minimal Riemannian
volume is:%
\begin{equation*}
dV_{bh}=2^{-4}d^{4}x.
\end{equation*}%
The Holmes-Thompson volume is given by:%
\begin{equation*}
\sigma _{ht}=\dfrac{1}{Vol\left( \mathbb{B}\right) }\underset{E_{x}^{t}}{%
\int }\dfrac{1}{2^{8}}d^{n}y=\dfrac{1}{2^{8}}\dfrac{Vol(E_{x}^{t})}{%
Vol\left( \mathbb{B}\right) }=2^{-4},
\end{equation*}%
(where we have used the equalities $\dfrac{Vol\left( E_{x}^{t}\right) }{%
Vol\left( \mathbb{B}\right) }=\left\vert \det (g^{t})\right\vert
^{-1/2}=2^{4}$). That is:%
\begin{equation*}
dV_{bh}=dV_{ht}=2^{-4}d^{4}x.
\end{equation*}

\bigskip

\textbf{Remark. }From (\ref{Cartan_vector_det}) and (\ref{det_B_M}), we find
out that the Cartan form $\mathbf{C}$ of $g$ identically vanishes - and yet, 
$g$ is non-Riemannian. This points out that positive definiteness and/or
smoothness of the metric are essential hypotheses for Deicke's theorem.

\subsection{\label{Bogoslovsky}A Bogoslovsky type metric}

Bogoslovsky metrics, expressible as: $L=\left( n_{i}y^{i}\right) ^{2b}\left(
\eta _{jk}y^{j}y^{k}\right) ^{1-b}$, where $b\in (0,1)$ and $n_{i}\in 
\mathbb{R}$ are covector components, are connected to very special
relativity, \cite{Bogoslovsky}. In the following, we will study a toy model
on $\mathbb{R}^{2},$ with $b=1/2$:%
\begin{equation}
L=y^{0}\sqrt{\left\vert (y^{0})^{2}-(y^{1})^{2}\right\vert }.
\label{Bogoslovsky_metric}
\end{equation}%
The metric tensor: 
\begin{equation*}
g(y)=\dfrac{1}{2\left\vert (y^{0})^{2}-(y^{1})^{2}\right\vert ^{3/2}}\left( 
\begin{array}{cc}
y^{0}[2(y^{0})^{2}-3(y^{1})^{2}] & (y^{1})^{3} \\ 
(y^{1})^{3} & -(y^{0})^{3}%
\end{array}%
\right)
\end{equation*}%
is only defined and invertible outside the lightlike directions $y^{1}=\pm
y^{0}.$ Its determinant:%
\begin{equation*}
\det g(y)=-\dfrac{2\left( y^{0}\right) ^{2}+\left( y^{1}\right) ^{2}}{%
4\left\vert (y^{0})^{2}-(y^{1})^{2}\right\vert }
\end{equation*}%
tends to minus infinity as we approach these axes - hence, we cannot prolong
it by continuity at $y^{1}=\pm y^{0}$; therefore, we will only determine in
this case the minimal Riemannian volume element $dV_{bh}.$

Critical directions $t=(t^{0},t^{1})$ for $\left\vert \det (g)\right\vert $
are $t^{0}=0$ and $t^{1}=0$. The former cannot be used as a time
orientation, since it is lightlike, that is, the only viable candidate for
the privileged time orientation is $t^{1}=0$ (and $t^{0}>0$). For this
direction, we find $\left\vert \det (g^{t})\right\vert =1/2$. Substituting
this value into the expression of $dV_{bh},$ we obtain the minimal
Riemannian volume element as:%
\begin{equation*}
dV_{bh}=2^{-1/2}d^{2}x.
\end{equation*}

\textbf{Acknowledgment. }Special thanks to Prof. Demeter Krupka for the
useful talks and advice, which led to the idea of minimal Riemannian volume
element exposed here.

\end{document}